\documentclass[12pt]{amsart}

\usepackage[arrow, matrix, curve]{xy}
\usepackage{amscd}
\usepackage{amssymb}
\usepackage{amsthm}
\usepackage{amsmath}
\usepackage{mathrsfs}

\usepackage{stmaryrd}

\newtheorem{theorem}{Theorem}[section]

\theoremstyle{plain}
\newtheorem{definition}[theorem]{Definition}

\newtheorem{prop}[theorem]{Proposition}

\newtheorem{lemma}[theorem]{Lemma}
\newtheorem{thm}[theorem]{Theorem}

\theoremstyle{definition}

\theoremstyle{remark}
\newtheorem{remark}[theorem]{Remark}
\newcommand{\ocs}{\overline{\mathcal{O}}_{\widetilde{S},\Lam}}
\newcommand{\oca}{\overline{\mathcal{O}}_{A,\Gamma}}
\newcommand{\OA}{\mathcal{O}_{A}}
\newcommand{\OC}{\mathcal{O}}

\newcommand{\W}{\mathcal{W}}
\newcommand{\wW}{\widehat{\mathcal{W}}}

\newcommand{\Z}{\mathbb{Z}}

\newcommand{\C}{\mathbb{C}}
\newcommand{\K}{\mathbb{K}}

\newcommand{\wb}{\widehat{\mathfrak{b}}}
\newcommand{\wh}{\widehat{\mathfrak{h}}}
\newcommand{\wg}{\widehat{\mathfrak{g}}}
\newcommand{\g}{\mathfrak{g}}
\newcommand{\lb}{\mathfrak{b}}
\newcommand{\h}{\mathfrak{h}}
\newcommand{\p}{\mathfrak{p}}
\newcommand{\J}{\mathcal{J}}
\newcommand{\st}{\widetilde{S}}
\newcommand{\wR}{\widehat{R}}
\newcommand{\ww}{\widehat{\W}}
\newcommand{\hc}{\widehat{\mathfrak{h}}_{\mathrm{crit}}^*}

\newcommand{\al}{{\alpha}}
\newcommand{\lam}{{\lambda}}
\newcommand{\Lam}{{\Lambda}}
\newcommand{\del}{{\Delta}}
\newcommand{\rdel}{{\overline{\Delta}}}
\newcommand{\cent}{\mathcal{Z}_{\widetilde{S}}(\Lam, \J)}
\newcommand{\cen}{\mathcal{Z}_{\widetilde{S}}(\Lam)}
\newcommand{\centa}{\mathcal{Z}_{A}(\Gamma, \J)}

\newcommand{\rp}{{\overline{P}}}

\newcommand{\Hom}{\mathrm{Hom}}
\newcommand{\End}{\mathrm{End}}

\numberwithin{equation}{section}

\begin{document}

\title[Centers for restricted category $\OC$]{Centers for the restricted category $\OC$ at the critical level over affine Kac-Moody algebras}

\author{Johannes K\"ubel*}
\address{Department of Mathematics, University of Erlangen, Germany}
\curraddr{Cauerstr. 11, 91058 Erlangen, Germany}
\email{kuebel@mi.uni-erlangen.de}
\thanks{*supported by the DFG priority program 1388}

%    General info

\date{\today}


\begin{abstract}
The restricted category $\OC$ at the critical level over an affine Kac-Moody algebra is a certain subcategory of the ordinary BGG-category $\OC$. We study a deformed version introduced by Arakawa and Fiebig and calculate the center of the deformed restricted category $\OC$.
\end{abstract}

\maketitle

\section{Introduction}  

The first step of Fiebig's proof of the Kazhdan-Lusztig conjecture for symmetrizable Kac-Moody algebras outside the critical hyperplanes was to calculate the center of a non-critical block of the deformed category $\OC$ (cf. \cite{4}). It turned out that it is isomorphic to the structure algebra of the moment graph associated to this block. Since the moment graph picture gives a description of the equivariant cohomology of the flag manifold corresponding to the Langlands dual Lie algebra, one gets the link to geometry to prove the Kazhdan-Lusztig conjecture.\\
In the present paper we use analogues arguments as in \cite{4} to describe centers of an appropriate categorical framework for the critical level representations over an affine Kac-Moody algebra. For this, let $\g\supset \lb \supset \h$ be a simple Lie algebra with a Borel and a Cartan subalgebra and $\wg$ be the affine Kac-Moody algebra associated to $\g$. Denote by $\wh$ the Cartan subalgebra corresponding to $\h$. The Feigin-Frenkel center acts on the category $\OC$ over $\wg$ as a graded algebra $Z=\bigoplus_{n\in \Z}Z_n$ generated by infinitely many homogeneous elements in the following way. Let $z \in Z_n$ be a homogeneous element of degree $n$. Then $z$ acts on $\OC$ as a transformation from the functor $(\cdot \otimes L(\delta))^n$ to the identity functor, where $L(\delta)$ is the one dimensional simple $\wg$-module with highest weight the smallest positive imaginary root. If $n \neq 0$, this action is only non-trivial on the subcategory $\OC_c \subset \OC$ of modules with critical level. In \cite{3} the authors introduce the restricted category $\overline{\OC}_c$ which consists of those modules $M \in \OC_c$ on which $Z_n$ acts trivially for all $n \in \Z \backslash \{0\}$. Furthermore, they describe the categorical structure of the subgeneric blocks of $\overline{\OC}_c$.\\
Now let $\Lam \subset \wh^*$ be a subset that parameterises the highest weights of simple modules in a block of $\OC_c$. Denote this block by $\overline{\OC}_\Lam$ and let $R(\Lam)$ be the subset of roots corresponding to $(\g,\h)$ which are integral with respect to $\Lam$. In this paper we work with a relative version of $\overline{\OC}_\Lam$. For this, let $\st$ be the localization of the symmetric algebra over $\h$. In chapter 2, we define the deformed block $\ocs$ as a relative version of $\overline{\OC}_\Lam$. We use the main results of \cite{3} to get the following:
The center of the category $\ocs$ is isomorphic to the $\st$-module
$$\left \{(z_\mu)_{\mu \in \Lam} \in \prod_{\substack{\mu \in \Lam}} \st \middle | z_\mu \equiv z_{\al \downarrow \mu} \, (\mathrm{mod} \, \al^\vee) \, \forall \al \in R(\Lam) \right \}$$
where $\al\downarrow \cdot: \Lam \rightarrow \Lam$ is a certain bijection on the set $\Lam$, which we define in chapter \ref{gsub}.\\
In this sense, this paper should be seen as an addendum to the first part of Fiebig's article \cite{4}. But unfortunately our result does not yet have a geometric counterpart as the non-critical situation of \cite{4} does.

\section{Preliminaries}

\subsection{Affine Kac-Moody algebras}

Let $\g \supset \lb \supset \h$ be a finite dimensional, complex, simple Lie algebra with a Borel and a Cartan subalgebra and denote by $\mathcal{L}(\g)= \g \otimes_\C \C[t,t^{-1}]$ the loop algebra. The Killing form $\kappa: \g \times \g \rightarrow \C$ allows us to take a central extension $\widetilde{\g}=\mathcal{L}(\g) \oplus \C K$ of the loop algebra, where $K$ is a central element. 
Adding an outer derivation operator $D=t \frac{\partial}{\partial t}$ to $\widetilde{\g}$ we get the affine Kac-Moody algebra $\wg$ corresponding to $\g$. As a vector space we have $\wg = (\g \otimes_\C \C[t,t^{-1}]) \oplus \C K \oplus \C D$ and the Lie bracket is given by:
$$[x \otimes t^n, y \otimes t^m]=[x,y] + n \delta_{m,-n} \kappa(x,y) K$$
$$[K,\wg]= \{0\}$$
$$[D,x \otimes t^n]=nx \otimes  t^n$$
where $x,y \in \g$. The corresponding affine Cartan subalgebra is given by $\wh =\h \oplus  \C K \oplus \C D$ and the affine Borel subalgebra by $\wb = (\g \otimes t \C[t] + \lb\otimes_\C \C[t]) \oplus \C K \oplus \C D$.\\
Denote by $R\supset R^+$ the root system with positive roots corresponding to the triple $(\g, \lb , \h)$. The projection $\wh =\h \oplus  \C K \oplus \C D\twoheadrightarrow \h$ induces an embedding $\h ^* \hookrightarrow \wh^*$. Let $\delta \in \wh^*$ be the imaginary root defined by $\delta(\h \oplus \C K) =\{0\}$ and $\delta(D) =1$. The inclusion $\h ^* \hookrightarrow \wh^*$ yields an embedding of $R$ into the affine root system of $(\wg,\wh)$ which we denote by $\wR$. It is given by $\wR=\wR^\mathrm{re}\cup \wR^\mathrm{im}$ where
$$\widehat{R}^{\mathrm{re}}= \{\al +n \delta\,|\,\al \in R\subset \widehat{R}, \, n \in \Z \}$$
$$\widehat{R}^{\mathrm{im}} =\{n \delta \,|\, n \in \Z, n \neq 0\}$$
Let $\theta \in R^+$ be the highest root and denote by $\Pi \subset R^+$ the set of simple roots. The set of affine simple roots $\widehat{\Pi}\subset \wR$ is then given by $\widehat{\Pi} = \Pi \cup \{-\theta +\delta\}$ and the set of positive affine roots by $\wR^+=R^+ \cup \{\al + n \delta \,|\, \al \in R, n>0\} \cup \{n \delta|n>0\}$.\\
We denote by $\W\subset \mathrm{Gl}(\h^*)$ the \textit{Weyl group} of the finite dimensional Lie algebra $\g$ and by $\widehat{\mathcal{W}} \subset \mathrm{Gl}(\wh^*)$ the \textit{affine Weyl group} of $\wg$. For any real root $\al \in \wR^\mathrm{re}$ the root space $\wg_\al$ is one dimensional and we can uniquely define the \textit{coroot} $\al^\vee \in \wh$ by the properties $\al ^\vee \in [\wg_\al, \wg_{-\al}]$ and $\langle \al,\al ^\vee\rangle=2$, where $\langle\cdot,\cdot\rangle:\wh^*\times\wh \rightarrow \C$ is the natural pairing. Let $\rho \in \wh^*$ be an element with $\rho(\al^\vee)=1$ for all coroots $\al^\vee$ associated to a simple affine root $\al \in \widehat{\Pi}$. The $\rho$-shifted dot-action of $\wW$ on $\wh^*$ is now given by 
$$w \cdot \lam =w(\lam + \rho)-\rho$$
where $w \in \widehat{\mathcal{W}}$ and $\lam \in\wh^*$. Although $\rho$ is not uniquely defined, the dot-action is independent of the choice of $\rho$.

\subsection{The deformed Category $\OC$}

Denote by $S=S(\h)$ the symmetric algebra over the complex vector space $\h$ and by $\widehat{S}=S(\wh)$ the symmetric algebra over $\wh$. Every $S$-algebra $A$ with structure map $\tau: S\rightarrow A$ carries an $\widehat{S}$-structure by composing $\tau$ with the map $\widehat{S} \twoheadrightarrow S$ induced by $\h^*\subset \wh^*$ from above. We call a commutative, associative, noetherian and unital $S$-algebra $A$ a \textit{deformation algebra}.\\
For a Lie algebra $\mathfrak{a}$ and a deformation algebra $A$ we define $\mathfrak{a}_A:= \mathfrak{a} \otimes_\C A$.
We can compose the structure morphism $\tau : S \rightarrow A$ to get a map $\wh \hookrightarrow S(\wh) \rightarrow S(\h) \rightarrow A$. We can extend this map to an $A$-linear map $\tau:\wh_A \rightarrow A$ which we denote by $\tau$ again and call it the \textit{canonical weight} corresponding to $A$. We can identify the $A$-dual space of $\wh_A$ with $\wh^*\otimes _\C A$ and use this identification to view every weight $\lam \in \wh^*$ as an element in $\wh_A^*$. This implies $\tau(D\otimes 1)=\tau(K\otimes 1)=0$.\\
For a $\wg_A$-module $M$ and a weight $\lam\in \wh^*$ we define the \textit{deformed} $\lam$-\textit{weight space}
$$M_\lam := \{v \in M \,|\, Hv = (\lam(H)+\tau(H))v \, \forall H \in \wh \}$$
We call $M$ a \textit{weight module} if it decomposes into the direct sum of its deformed weight spaces. 

\begin{definition}
The deformed category $\OA$ is the full subcategory of $\wg_A\mathrm{-mod}$ which consists of all modules $M$ which are
\begin{enumerate}
\item weight modules, and 
\item locally $\wb_A$-finite, i.e., for any $m \in M$ the submodule $\lb_A m$ is finitely generated over $A$.
\end{enumerate}
\end{definition}

The \textit{deformed Verma module} with highest weight $\lam \in \wh^*$ is an element of $\OA$ and is defined by
$$\Delta_A(\lam):= U(\wg_A) \otimes_{U(\wb_A)} A_\lam$$
where $A_\lam$ is $A$ as an $A$-module and the left action of $\wb$ is given by the composition $\wb \twoheadrightarrow \wh \stackrel{\lam+\tau}{\rightarrow} A$.\\
For a homomorphism $A\rightarrow A'$ of $S$-algebras we get a functor
$$\cdot \otimes_A A' : \OA \longrightarrow \OC_{A'}$$
which maps $\del_A(\lam)$ to $\del_{A'}(\lam)$.
\begin{lemma}[\cite{4}, Proposition 2.1]
Let $A$ be a local deformation algebra with residue field $\K$. Then the base change functor $\cdot \otimes_A \K$ induces a bijection between isomorpism classes of simple objects in $\OA$ and the ones in $\OC_\K$.
\end{lemma}
This lemma implies that the simple modules of $\OA$ and $\OC_\K$ are parameterized by $\wh^*$ (cf. \cite{3}, 2.8). We denote the simple module in $\OA$ ($\OC_\K$, resp.) with highest weight $\lam \in \wh^*$ by $L_A(\lam)$ ($L_\K(\lam)$, resp.).\\
In contrast to the category $\OC$ over the simple Lie algebra $\g$ we do not have projective covers of simple objects in the category $\OC$ over $\wg$ in general. However, in certain truncated subcategories projective covers do exist.

\begin{definition}
\begin{enumerate}
\item A subset $\J \subset \wh^*$ is called open, if for all $\lam \in \J$ and $\mu \in \wh^*$ with $\mu \leq \lam$ we have $\mu \in \J$.
\item We call a subset $\J \subset \wh^*$ \textit{bounded from above} if for any $\lam \in \J$ the set $\J \cap \{\geq \lam \} = \{\mu \in \J \,|\, \mu \geq \lam \}$ is finite.
\end{enumerate}
\end{definition} 

\begin{definition}
Let $\J \subset \wh^*$ be open and $M \in \OA$.
\begin{enumerate}
\item We define $M^\J := M/(U(\wg_A) \bigoplus_{\lam \in \wh^* \backslash \mathcal{J}} M_\lam$).
\item We define $\OA^\J \subset \OA$ to be the full subcategory of objects $M$ with $M= M^\J$.
\end{enumerate}
\end{definition}
It is easy to see that $M \mapsto M^\J$ induces a functor $\OA \rightarrow \OA^\J$ which is left adjoint to the inclusion functor $\OA^\J \subset \OA$. Furthermore, $\Delta_A(\lam)$ and $L_A(\lam)$ are contained in $\OA^\J$ if and only if $\lam \in \J$.

\begin{thm}[\cite{3}, Theorem 2.8]\label{thmpr}
Let $\J$ be a bounded, open subset of $\wh^*$ and $A$ a local deformation algebra with residue field $\K$.
\begin{enumerate}
\item For each $\lam \in \J$ there exists a projective cover $P_A^\J(\lam)$ of $L_A(\lam)$ in $\OA^\J$. It admits a Verma flag, i.e., a filtration with layers isomorphic to deformed Verma modules, and we have
$$(P_A^\J(\lam): \Delta_A(\mu)) =\begin{cases} [\Delta_\K (\mu):L_\K(\lam)], \, \mathrm{if} \, \mu \in \J, \\ 0, \,\,\,\,\,\,\,\,\,\,\mathrm{otherwise} \end{cases}$$
where the left side denotes the multiplicity of $\Delta_A(\mu)$ in a Verma flag and the right side the multiplicity of $L_\K(\lam)$ as a simple subquotient of $\Delta_\K(\mu)$.
\item If $\J' \subset \J$ is open as well, then
$$P_A^\J(\lam)^{\J'} \cong P_A^{\J'}(\lam).$$
\item If $A \rightarrow A'$ is a homomorphism of local deformation algebras and $P \in \OA^\J$ is projective, then $P \otimes_A A'$ is projective in $\OC_{A'}^\J$.
\item We have $P_A^\J(\lam) \otimes_A \K \cong P_\K^\J(\lam)$.
\item Let $P$ be a finitely generated projective object in $\OA^\J$ and $A \rightarrow A'$ be a homomorphism of local deformation algebras. For any $M \in \OA^\J$ the natural map
$$\Hom_{\OA}(P,M) \otimes_A A' \longrightarrow \Hom_{\OC_{A'}}(P \otimes_A A', M \otimes_A A')$$
is an isomorphism.
\end{enumerate}
\end{thm}

For $A$ a local deformation algebra with residue field $\K$, we introduce an equivalence relation on $\wh^*$. Let $\lam,\mu \in \wh^*$. We call $\mu$ \textit{$A$-larger} than $\lam$ if $[\Delta_\K(\mu): L_\K(\lam)] \neq0$. Then, $\sim_A$ is defined to be the equivalence relation on $\wh^*$ generated by pairs $(\lam, \mu)$, s.t. $\mu$ is $A$-larger than $\lam$. Let $\Lam \in \wh^* / \sim_A$. $\OC_{A, \Lam}$ is defined to be the full subcategory of $\OA$ that consists of all objects $M$ with the property that any simple subquotient of $M$ has highest weight in $\Lam$.

\begin{thm}[\cite{4}, Proposition 2.8]
The functor
\begin{eqnarray*}
		\begin{array}{ccc}
		\bigoplus\limits_{\Lambda \in \wh^{*}/\sim_A} \OC_{A,\Lambda}  &\longrightarrow&  \OA\\
		(M_\Lambda)_{\Lambda \in \wh^{*}/\sim_A}& \longmapsto & \bigoplus\limits_{\Lambda \in \wh^{*}/\sim_A} M_\Lambda
		
		\end{array}
	\end{eqnarray*}
is an equivalence.
\end{thm}

\section{Restricted representations}

The main purpose of this chapter is to introduce the restricted deformed category $\OC$. It is introduced in \cite{3} and is a natural categorical framework to study the highest weight modules at the critical level.

\subsection{The critical level}
Let $M \in \OA$ and $\lam \in \wh^*$. The central element $K \in \wg$ acts on the weight space $M_\lam$ by $\lam(K)\in \C$. We denote by $M_k$ the eigenspace of the action of $K$ with eigenvalue $k \in \C$. If $M=M_k$ we call $M$ a \textit{module of level $k$}. The category $\OA$ decomposes into the direct sum of subcategories of modules with equal level, i.e., $\OA=\bigoplus_{k\in \C} \OC_{A,k}$, where $\OC_{A,k}$ consists of those modules which are of level $k$.\\
Let $(\cdot|\cdot): \wg \times \wg \rightarrow \C$ be the non-degenerate, symmetric and invariant bilinear form on $\wg$ which is induced by the Killing form on $\g$. Its restriction to $\wh \times \wh$ is non-degenerate as well and induces a non-degenerate bilinear form $(\cdot|\cdot)$ on $\wh^*$. For $A$ a local deformation algebra with canonical weight $\tau:\wh_A \rightarrow A$, we can extend it to $(\cdot|\cdot)_A: \wh_A^* \times \wh_A^* \rightarrow A$, the $A$-linear continuation of $(\cdot|\cdot)$. Recall the element $\rho \in \wh^*$ with $\rho(\al^\vee)=1$ for any simple root $\al \in \widehat{\Pi}$. Note that for the imaginary root $\delta \in \wR$ we have $(\delta|\delta)=0$.

\begin{definition}
For $\lam \in \wh^*$ we define the set of \textit{integral roots} (with respect to $\lam$ and $A$) by
$$\wR_A (\lam) := \{ \al \in \wR \,|\, 2(\lam+\rho+\tau|\al)_A \in \Z(\al|\al)_A \}$$
The corresponding \textit{integral Weyl group} is defined by
$$\ww_A(\lam):= \langle s_\al \,|\, \al \in \wR_A(\lam) \cap \wR^{\mathrm{re}} \rangle \subset \ww$$
\end{definition}

For an equivalence class $\Lam \in \wh^*/\sim_A$ and $\mu, \lam \in \Lam$, we have $\lam(K)=\mu(K)$ and also $\wR_A(\lam)=\wR_A(\mu)$.

The complex number $\mathrm{crit}:= -\rho(K)$ is called \textit{critical level}.

\begin{lemma}[\cite{2}, Lemma 4.2]
For $\Lam \in \wh^*/\sim_A$ the following are equivalent.
\begin{enumerate}
\item $\Lam$ is critical, i.e., $\lam(K)=\mathrm{crit}$ for all $\lam \in \Lam$.
\item We have $\lam + \delta \in \Lam$ for all $\lam \in \Lam$.
\item We have $n \delta \in \wR_A(\lam)$ for all $n \neq0$ and all $\lam \in \Lam$.
\end{enumerate}
\end{lemma}

\subsection{Restricted representations}

We recall the action of the Fei\-gin-Fren\-kel center on $\OA$ as explained in \cite{2} and \cite{3}. Let $A$ be a local deformation algebra with residue field $\K$. In chapter 3.1 of \cite{3} the authors introduce an equivalence

$$T: \OA \longrightarrow \OA$$
defined by $T:= \cdot \otimes_\C L(\delta)$ where $L(\delta)$ is the one dimensional simple $\wg$-module with highest weight $\delta$. The inverse functor of $T$ is given by $T^{-1} = \cdot \otimes_\C L(-\delta)$. Denote by $T^n$ the $n$-fold composition of $T$. 

\begin{lemma}[\cite{3}, Chapter 3.1]
A block $\OC_{A,\Lam}$ with $\Lam \in \wh^*/\sim_A$ is preserved by the functor $T$ if and only if $\Lam$ is of level $k(\Lam)= \mathrm{crit}$.
\end{lemma}
Let $V^\mathrm{crit}(\g)$ be the universal affine vertex algebra associated with $\g$ at the critical level and denote by $\mathfrak{z}$ its center. On any $M \in \OC_{A,\mathrm{crit}}$, $\mathfrak{z}$ induces the structure of a graded module over the polynomial ring (of infinite rank)
$$Z:= \bigoplus_{n \in \Z} Z_n = \C[p_s^{(i)}, \, i=1,..., \mathrm{rk} \g, \, s \in \Z].$$
Any $z \in Z_n$ for $n \in \Z$ acts on $\OA$ as a natural transformation from $T^n$ to the identity functor on $\OA$. Thus, the action of $z$ on $M$ is given by a homomorphism $z^M:T^n M\rightarrow M$. This action respects base change, i.e., for $A\rightarrow A'$ a homomorphism of deformation algebras, the base change functor $\cdot\otimes_A A':\OC_{A,\mathrm{crit}}\rightarrow \OC_{A',\mathrm{crit}}$ induces a natural isomorphism $z^{M\otimes_A A'} \cong z^M \otimes_A \mathrm{id}_{A'}$.

\begin{definition}
Let $M \in \OC_{A,\mathrm{crit}}$. We call $M$ \textit{restricted} if for all $n \neq 0$ and all $z \in Z_n$, the homomorphism $z^M: T^n M \rightarrow M$ is zero.
Then the \textit{restricted deformed category} $\overline{\OC}_{A, \mathrm{crit}}$ is the full subcategory of $\OC_{A,\mathrm{crit}}$ of all restricted modules. 
\end{definition}

For $\J \subset \wh^*$ open and bounded let $\overline{\OC}_{A,\mathrm{crit}}^\J := \overline{\OC}_{A,\mathrm{crit}} \cap \OC_A^\J$. Denote by $Z_n M$ the submodule of $M$ generated by the subset $\{z^M(m)\,|\, m \in T^nM, z\in Z_n\}\subset M$. The functor $(\cdot)^\mathrm{res}:\OC_{A, \mathrm{crit}} \rightarrow \overline{\OC}_{A,\mathrm{crit}}$ given by
$$M \rightarrow M^\mathrm{res}:= M/\sum_{n \in \Z \backslash \{0\}} Z_n M$$
is well defined and left adjoint to the inclusion functor $\overline{\OC}_{A, \mathrm{crit}} \subset \OC_{A, \mathrm{crit}}$. 
\begin{definition}
Let $\lam \in \wh^*$ be of critical level. The restricted Verma module is defined by
$$\rdel_A(\lam):= \del_A (\lam)^{\mathrm{res}}$$
\end{definition}

We collect some results which can be found in \cite{3} and which we need for the calculation of the center of a critical restricted block.
\begin{lemma}
Restricted Verma modules respect base change, i.e., for any homomorphism $A \rightarrow A'$ of deformation algebras we have $\rdel_A(\lam) \otimes_A A' \cong \rdel_{A'}(\lam)$.
\end{lemma}

\begin{proof}
The claim is obvious for non-restricted Verma modules. Since the action of the center respects base change it is also true for restricted Verma modules.
\end{proof}

\begin{lemma}[\cite{3}, Lemma 3.8, Lemma 3.3, Lemma 3.4]\label{lem2}
Let $\lam \in \wh^*$ be critical, $\J \subset \wh^*$ open, $A \rightarrow A'$ a homomorphism of local deformation algebras and $M \in \OA$. Then 
\begin{enumerate}
\item for any $\mu \in \wh^*$ and $\K$ the residue field of $A$ the weight space $\rdel_A(\lam)_\mu$ is a free $A$-module with
$$\mathrm{rk}_A \rdel_A(\lam)_\mu = \mathrm{dim}_\K \rdel_\K(\lam)_\mu,$$
\item $(M^{\mathrm{res}})^\J \cong (M^\J)^\mathrm{res}$,
\item the canonical map
$$(M \otimes_A A')^\mathrm{res} \stackrel{}{\rightarrow} M^\mathrm{res} \otimes_A A'$$
is an isomorphism.
\end{enumerate}
\end{lemma}

\begin{remark}
Note that in \cite{3}, Lemma 3.4, the isomorphism in (3) of the above lemma is formulated by
$$(M \otimes_A A')^\mathrm{res} \stackrel{\sim}{\rightarrow} (M^\mathrm{res} \otimes_A A')^\mathrm{res}$$
But since $M^\mathrm{res} \otimes_A A'$ is an object of $\overline{\OC}_{A'}$ already, we do not have to restrict it anymore.
\end{remark}

\subsection{Restricted projective objects}

We recall the construction of projective covers from \cite{3}. Let $A$ be a local deformation algebra with residue field $\K$, $\J \subset \wh^*$ an open, bounded subset and let $\lam \in \J$ be of critical level. Recall the projective cover $P_A ^\J (\lam) \twoheadrightarrow L_A(\lam)$. Define $\overline{P}_A^\J(\lam):=P_A^\J(\lam)^\mathrm{res}$. We say, a module $M \in \overline{\OC}_{A,\mathrm{crit}}$ has a \textit{restricted Verma flag}, if it has a finite filtration with subquotients isomorphic to restricted Verma modules. As in the non-restricted case, we denote by $(M:\rdel_A(\lam))$ the multiplicity of $\rdel_A(\lam)$ as a subquotient in a restricted Verma flag of $M$. 

\begin{thm}[\cite{3}, Theorem 4.9, and \cite{5}, Theorem 4.3]\label{thm1}
Let $A$ be a local deformation algebra with residue field $\K$, $\J$ an open and bounded subset of $\wh^*$ and $\mu \in \J$ be of critical level. Then $\rp_A^\J(\mu)$ admits a restricted Verma flag and we have
$$(\rp_A^\J(\mu):\rdel_A(\lam))=\begin{cases} [\rdel_\K(\lam):L_\K(\mu)], \,\,\,\,\mathrm{if} \, \lam \in \J \\ 0, \,\,\,\,\,\,\,\mathrm{otherwise} \end{cases}$$
for all $\lam \in \wh^*$. Furthermore, $\rp_A^\J(\mu)$ is a projective cover of $L_A(\mu)$ in $\overline{\OC}_A^\J$.
\end{thm}

For the method of proof we are using for calculating the center of a restricted critical block, the following lemma will be quite important.

\begin{lemma}\label{lem3}
Let $A \rightarrow A'$ be a homomorphism of local deformation algebras. Let $M, \rp \in \overline{\OC}_{A}^\J$ and $\rp$ be a finitely generated projective object. Then the canonical map
$$\Hom_{\overline{\OC}_A}(\rp,M)\otimes_A A' \longrightarrow \Hom_{\overline{\OC}_{A'}}(\rp \otimes_A A',M\otimes_A A')$$
is an isomorphism of $A'$-modules.
\end{lemma}

\begin{proof}
By Theorem \ref{thm1} we find a finitely generated projective object $P \in \OA^\J$, s.t. $P^\mathrm{res} \cong \rp$. Since $(\cdot)^{\mathrm{res}}$ is left adjoint to the inclusion functor $\overline{\OC}_A^\J \hookrightarrow \OA^\J$ and since $M\otimes_A A'$ is an object of $\overline{\OC}_{A'}$, we get 
$$\Hom_{\overline{\OC}_A}(P^{\mathrm{res}},M)\otimes_A A' \stackrel{\sim}{\longrightarrow} \Hom_{{\OC}_A}(P,M)\otimes_A A'$$
and 
$$\Hom_{\overline{\OC}_{A'}}((P \otimes_A A')^{\mathrm{res}},M\otimes_A A') \stackrel{\sim}{\longrightarrow} \Hom_{{\OC}_{A'}}(P \otimes_A A',M\otimes_A A')$$
By Lemma \ref{lem2}, we have $(P \otimes_A A')^{\mathrm{res}} \cong P^{\mathrm{res}}\otimes_A A'$. Thus, the claim follows from Theorem \ref{thmpr} (5).
\end{proof}

\begin{remark}
The Feigin-Frenkel conjecture claims that for $\lam, \mu \in \J$ the multiplicities $(\rp_A^\J(\mu):\rdel_A(\lam))$ are given by certain periodic Kazhdan-Lusztig polynomials evaluated at one. These periodic polynomials depend on the relative position between $\lam$ and $\mu$, and if $\lam$ and $\mu$ are "far away" from each other these polynomials are zero. Thus, if the Feigin-Frenkel conjecture was true, the projective cover $\rp_A^\J(\mu)$ would stabilize for $\J$ big enough and $\rp_A^\J(\mu)$  would also be a projective cover of $L_A(\mu)$ in the bigger non-truncated category $\overline{\OC}_A$.
\end{remark}

\subsection{The restricted block decomposition}

We set 
$$\hc:= \{\lam \in \wh^*\,|\,\lam(K)=\mathrm{crit}\}.$$
We define an equivalence relation $\sim_A^\mathrm{res}$ on the critical hyperplane $\hc$. Let $\lam, \mu \in \hc$. We say, $\mu$ is \textit{restricted $A$-larger} than $\lam$ if $L_A(\lam)$ appears as a subquotient of $\rp^\J_A(\mu)$ for $\J\subset \wh^*$ big enough. $\sim_A^\mathrm{res}$ is then defined to be the equivalence relation on $\hc$ that is generated by pairs $(\lam,\mu)$, s.t. $\mu$ is restricted $A$-larger than $\lam$. For an equivalence class $\Lam \in \hc/\sim_A^\mathrm{res}$ let $\overline{\OC}_{A,\Lam} \subset \overline{\OC}_{A,\mathrm{crit}}$ be the full subcategory of objects $M$, such that $[M:L_A(\lam)]\neq 0$ implies $\lam \in \Lam$.

\begin{thm}[\cite{3}, Theorem 5.2]
The functor
\begin{eqnarray*}
		\begin{array}{ccc}
		\bigoplus\limits_{\Lambda \in \hc/\sim_A^\mathrm{res}} \overline{\OC}_{A,\Lambda}  &\longrightarrow&  \overline{\OC}_{A,\mathrm{crit}}\\
		(M_\Lambda)_{\Lambda \in \hc/\sim_A^\mathrm{res}}& \longmapsto & \bigoplus\limits_{\Lambda \in \wh^{*}/\sim_A^\mathrm{res}} M_\Lambda
		
		\end{array}
	\end{eqnarray*}
is an equivalence.
\end{thm}

We want to recall a more detailed description of the restricted critical blocks. Denote by $\overline{\cdot}: \wh^* \rightarrow \h^*, \lam \mapsto \overline{\lam}$ the projection with respect to the decomposition $\wh = \h \oplus \C D \oplus \C K$ and denote by $\overline{\Lam} \subset \h^*$ the image of a subset $\Lam \subset \wh^*$.\\
For $\Lam \in \hc / \sim_A^\mathrm{res}$ we define the \textit{finite integral root system} (with respect to $\Lam$ and $A$) by 
$$R_A(\Lam) := \{\al \in R \, | \, 2(\lam+\rho+\tau| \al)_A \in \Z(\al| \al)_A \, \mathrm{for \,all} \, \lam \in \Lam \}$$
and the \textit{finite integral Weyl group} by
$$\W_A(\Lam) := \langle s_\al \, |\, \al \in R_A(\Lam)\rangle \subset \W$$

\begin{lemma}[\cite{3}, Lemma 5.3]\label{lemweyl}
Let $\Lam \in \hc / \sim_A^\mathrm{res}$ be a critical restricted equivalence class. Then, for all $\lam \in \Lam$
$$\overline{\Lam}=\W_A(\Lam)\cdot\overline{\lam}.$$
\end{lemma}

\subsection{The generic and subgeneric cases}\label{gsub}

For the calculation of the center we will need the description of the generic and subgeneric equivalence classes (cf. \cite{2}).\\
In the rest of this paper, $\st$ will be the localization of $S=S(\h)$ at the maximal ideal generated by $\h$. If $\p \subset \st$ is a prime ideal, we denote by $S_\p$ the localization of $\st$ at $\p$.\\
We need some more notation: Let $\Lam \in \hc/\sim_A^\mathrm{res}$. For a root $\al \in R_A(\Lam)$ and $\lam \in \Lam$ we define $\al \downarrow \lam$ (resp. $\al \uparrow \lam$) to be the element in the set $\{s_\al \cdot \lam, s_{-\al+\delta}\cdot \lam\}$ which is smaller (resp. larger) than or equal to $\lam$. We then define inductively $\al \downarrow^n \lam:=\al \downarrow(\al \downarrow ^{n-1}\lam)$ and $\al \uparrow^n \lam:=\al \uparrow(\al \uparrow ^{n-1}\lam)$.
 
\begin{definition}
Let $\Lam \in \hc / \sim_A^\mathrm{res}$ be a critical restricted equivalence class. We call $\Lam$
\begin{enumerate}
\item \textit{generic} if $\overline{\Lam} \subset \h^*$ consists of one element,
\item \textit{subgeneric} if $\overline{\Lam} \subset \h^*$ consists of two elements.
\end{enumerate}
\end{definition}

\begin{lemma}[\cite{3}, Lemma 5.5]\label{lemsub}
Let $\p \subset \st$ be a prime ideal of height one and let $\Lam \subset \hc$ be an equivalence class for $\sim_{S_\p}^\mathrm{res}$.
\begin{enumerate}
\item If $\al ^\vee \notin \p$ for all $\al \in R$, then $\Lam$ is generic.
\item If $\al ^\vee \in \p$ for some $\al \in R$, then $R_{S_\p}(\Lam) \subset \{\al, -\al\}$ and $\Lam$ is either generic or subgeneric.
\end{enumerate}
\end{lemma}

We finish this chapter with the main results of \cite{3}.

\begin{thm}[\cite{3}, Theorem 5.6]\label{thm3}
Let $\Lam \in \hc / \sim_A^\mathrm{res}$, $\lam \in \Lam$ and $\J \subset \hc$ an open and bounded subset.
\begin{enumerate}
\item Suppose that $\Lam$ is generic. Then $\Lam$ consists of one element and
$$\rp_A^\J(\lam) \cong \rdel_A(\lam)$$
if $\lam \in \J$.
\item Suppose that $\Lam$ is subgeneric and that $\overline{\Lam}=\{\overline{\lam}, s_\al \cdot \overline{\lam}\}$ for some $\al \in R$. Then there is a non-split short exact sequence
$$0\rightarrow \rdel_A(\al \uparrow \lam) \rightarrow \rp_A^\J(\lam) \rightarrow \rdel_A(\lam) \rightarrow 0$$
and a short exact sequence
$$0\rightarrow L_A(\al \downarrow \lam) \rightarrow \rdel_A(\lam) \rightarrow L_A(\lam) \rightarrow 0$$
if $\J$ contains $\lam$ and $\al \uparrow  \lam$.
\item If $\Lam$ is subgeneric with $R_A(\Lam)=\{\pm \al\}$ and $\lam \in \Lam$, we have 
$$\Lam = \{...,\al \downarrow^2 \lam, \al \downarrow \lam, \lam, \al \uparrow \lam, \al \uparrow^2 \lam,...\}.$$
\end{enumerate}
\end{thm}
The second short exact sequence in (2) follows from the first one and BGGH-reciprocity.

\section{The center of $\ocs$ for $\Lam$ critical}

Recall that the center of a category is the ring of endo-transformations of the identity functor. Let $A$ be a deformation algebra and fix a critical equivalence class $\Lam \in \hc / \sim_{A}^{\mathrm{res}}$. For an open, bounded subset $\J \subset \hc$ denote by $\mathcal{Z}_A(\Lam, \J)$ the center of $\overline{\OC}_{A,\Lam}^\J$ and by $\mathcal{Z}_A(\Lam)$ the center of $\overline{\OC}_{A,\Lam}$. We first consider the case $A=\st$, the localization of $S$ at $S\h$. Thus, we have
$$\cent := \mathcal{Z}(\ocs^\J) = \End (\mathrm{id}_{\ocs^\J})$$

Since in general we only have enough projective objects in the truncated categories, we have to express the center of $\ocs$ as a limit of the centers $\mathcal{Z}_A(\Lam, \J)$ which runs over open and bounded subsets $\J \subset \hc$. The main result is then
\begin{thm}
Let  $\st$ be the localization of $S$ at the maximal ideal generated by $\h$ and $\Lam \in \hc/ \sim_{\st}^\mathrm{res}$. Then we have an isomorphism of $\st$-modules
$$\cen \cong \left \{ (z_\mu)_{\mu \in \Lam} \in \prod_{\substack{\mu \in \Lam}} \st \, \middle| \, z_\mu \equiv z_{\al \downarrow \mu} \, (\mathrm{mod}\, \al ^\vee) \, \forall \, \al \in R_{\st}(\Lam) \right \}$$
\end{thm}
We use Lemma \ref{lem3} and a localization process to split the problem into generic and subgeneric situations. Following \cite{4}, we first relate the center to the endomorphism rings of a generating set of restricted projective objects.
\begin{remark}\label{rmk2}
Let $A$ be a localization of $\st$ at a prime ideal $\p \subset \st$ and $\Gamma$ an equivalence class under $\sim_{A}^{\mathrm{res}}$. The block $\oca^\J$ is generated by the set of projective covers $\{\rp^\J_A(\lam)\}_{\lam \in \Gamma\cap\J}$. By the same arguments as given in \cite{4}, chapter 3.1, we get that evaluating on indecomposable projective objects induces an injective map
$$\centa \hookrightarrow \prod_{\substack{\mu \in \Gamma\cap\J}} \End_{\overline{\OC}_A}( \rp^\J_A(\mu))$$

and the image of this map is given by the subset

\begin{equation*}
\begin{split}
{\Biggl\{}(z_\mu)_{\mu \in \Gamma\cap\J} \in \prod_{\substack{\mu \in \Gamma\cap\J}}\End_{\overline{\OC}_A}( \rp^\J_A(\mu))  {\Biggl|} z_{\mu} \circ f = f \circ z_\lam \,\\ \forall \, \lam, \mu \in \Gamma\cap\J, \, f \in \Hom_{\overline{\OC}_A} (\rp^\J_A (\lam), \rp^\J_A(\mu))  {\Biggl \}}
\end{split}
\end{equation*}
\end{remark}

Let $A\rightarrow A'$ be a homomorphism of local deformation algebras. The base change, Lemma \ref{lem3}, for the endomorphism rings of restricted projective objects induces a map
$$\centa \longrightarrow \mathcal{Z}_{A'}(\Gamma, \J).$$

For another open, bounded subset $\J'\subset \wh^*$ with $\J \subset \J'$ the natural maps $\rp_A^{\J'}(\mu) \rightarrow \rp_A^{\J}(\mu)$ for $\mu \in \Gamma \cap \J'$ induce a map $\mathcal{Z}_A(\Gamma, \J') \rightarrow \centa$. Since all finitely generated modules of $\oca$ already lie in a subcategory $\oca^\J$ for a certain open and bounded subset $\J \subset \wh^*$, and since the center is already uniquely defined by its action on the finitely generated objects, we have
$$\mathcal{Z}_{A}(\Gamma) \cong \lim_{\stackrel{\longleftarrow}{\J}} \centa$$
The base change maps $\centa \rightarrow \mathcal{Z}_{A'}(\Gamma, \J)$ then induce a base change map $\mathcal{Z}_A(\Gamma) \rightarrow \mathcal{Z}_{A'}(\Gamma)$.

\begin{lemma}
Let $A$ be a localization of $\st$ at a prime ideal $\p \subset \st$. The evaluation on restricted Verma modules induces an injective map
$$\mathcal{Z}_A(\Gamma) \hookrightarrow \prod_{\substack{\mu \in \Gamma}}\End_{\overline{\OC}_A}( \rdel_A(\mu)) \cong \prod_{\substack{\mu \in \Gamma}} A$$
\end{lemma}

\begin{proof}
Let $Q:= Q(A)$ be the quotient field of $A$. Then by Theorem \ref{thm3} (1), all Verma modules $\rdel_Q(\lam)$ are projective and we have  $\Hom_{\overline{\OC}_Q}(\rdel_Q(\mu), \rdel_Q(\lam)) = 0$ for $\mu \neq \lam$. But by the description of the center from above and by base change for the center we get a commutative diagram

\begin{eqnarray*}
	\begin{CD}
   \mathcal{Z}_A(\Gamma)   @>>> \prod_{\substack{\mu \in \Gamma}}\End_{\overline{\OC}_A}( \rdel_A(\mu)) \cong \prod_{\substack{\mu \in \Gamma}} A\\
   @VV V @VVV\\
  \mathcal{Z}_Q(\Gamma) @>\sim >> \prod_{\substack{\mu \in \Gamma}}\End_{\overline{\OC}_Q}( \rdel_Q(\mu)) \cong \prod_{\substack{\mu \in \Gamma}} Q
	\end{CD}
\end{eqnarray*}

where the lower horizontal is an isomorphism and the verticals are injective. But then the upper horizontal is injective as well. 
\end{proof}

\begin{remark}
We want to describe the image of the map
$$\mathcal{Z}_{\st}(\Gamma) \hookrightarrow \prod_{\substack{\mu \in \Gamma}}\st$$

The strategy to describe the image of this map is to localize all appearing modules at prime ideals $\p$ of height one. So let $\p\subset \st$ be such an ideal. If $\al^\vee \notin \p$ for all $\al \in R_{\st}(\Gamma)$, all $\rdel_{S_\p}(\mu)$ are projective and we get $\mathcal{Z}_{S_\p}(\Gamma) \stackrel{\sim}\rightarrow \prod_{\substack{\mu \in \Gamma}} S_\p$ which is the generic situation. We will deal with the subgeneric case in the next chapter.
\end{remark}

\subsection{The subgeneric case}

Let $S_\al$ be the localization of $\widetilde{S}$ at the prime ideal generated by $\al ^\vee$. We fix an equivalence class $\Gamma \subset \hc$ under $\sim_{S_\al}^\mathrm{res}$ which is not generic. For $\lam \in \Gamma$ we get, by Lemma \ref{lemsub} and Theorem \ref{thm3} (3), $\Gamma= \wW_\al \cdot \lam=\{...,\al \downarrow \lam, \lam, \al \uparrow \lam, \al \uparrow^2 \lam,...\}$, where $\wW_\al\subset \wW$ is the affine subgroup generated by the reflections $s_{\al +n \delta}$ with $n \in \Z$. Recall that by Theorem \ref{thm3} we have a short exact sequence
$$\rdel_{S_\al}(\al \uparrow \lam) \hookrightarrow \rp_{S_\al}^\J(\lam) \twoheadrightarrow \rdel_{S_\al}(\lam)$$
if $\al \uparrow \lam, \lam \in \J$. This implies that for $\J'\supset \J$ we have $\rp_{S_\al}^{\J'}(\lam)\cong \rp_{S_\al}^\J(\lam)$ in $\overline{\OC}_{S_\al}$. Thus, for any $\mu \in \Gamma$ we will always assume that the corresponding open, bounded subset $\J_\mu$ is big enough, such that we can write $\rp_{S_\al}(\mu)=\rp_{S_\al}^{\J_\mu}(\mu)$.

\begin{lemma}
Restriction to the restricted Verma module $\rdel_{S_\al}(\al \uparrow \mu) \hookrightarrow \rp_{S_\al}(\mu)$ for every $\mu \in \Gamma$ induces a surjective map
\begin{displaymath}
\begin{array}{ccc}
\prod_{\substack{\mu\in \Gamma}} \End_{\overline{\OC}_{S_\al}}(\rp_{S_\al}(\mu)) & \twoheadrightarrow & \prod_{\substack{\mu\in \Gamma}} \End_{\overline{\OC}_{S_\al}}(\rdel_{S_\al}(\mu)) \\
(f_\mu)_{\mu \in \Gamma} & \mapsto & (f_{\mu} |_{\rdel_{S_\al}(\al \uparrow \mu)})_{\mu \in \Gamma}
\end{array}
\end{displaymath}
\end{lemma}

\begin{proof}
Since $\End_{\overline{\OC}_{S_\al}}(\rdel_{S_\al}(\mu)) = S_\al \cdot \mathrm{id}_{\rdel_{S_\al}(\mu)}$, every endomorphism of the restricted Verma module lifts to an endomorphism of $\rp_{S_\al}(\al \downarrow\mu)$.
\end{proof}

Identifying $\prod_{\substack{\mu \in \Gamma}} \End_{\overline{\OC}_{S_\al}}(\rdel_{S_\al}(\mu))$ with $\prod_{\substack{\mu \in \Gamma}} S_{\al}$ and using the natura\-lity of the action of the center, we get a commutative diagram\\

\begin{displaymath}
    \xymatrix{
       \mathcal{Z}_{S_\al}(\Gamma)  \ar@{^{(}->}[dr] \ar@{^{(}->}[rr] &   &  \prod_{\substack{\mu \in \Gamma}} S_{\al} \\
                                & \prod_{\substack{\mu \in \Gamma}} \End_{\overline{\OC}_{S_\al}}(\rp_{S_\al}(\mu)) \ar@{->>}[ur]&  }
\end{displaymath}
The aim is now to describe the image of the composition of the down- with the up-going arrow in this diagram.

\begin{prop}
Let $\lam, \mu \in \hc$, then $\Hom_{\overline{\OC}_{S_\al}}(\rp_{S_\al}(\mu),\rp_{S_\al}(\lam))$ is a free $S_\al$-module and we have
\begin{equation*}
\mathrm{rk}_{S_\al} \Hom_{\overline{\OC}_{S_\al}}(\rp_{S_\al}(\mu),\rp_{S_\al}(\lam))=
\begin{cases}
1, \mathrm{if} \, \mu = \al \downarrow \lam \, \mathrm{or} \, \mu =\al \uparrow \lam \\
2, \mathrm{if} \, \mu = \lam \\
0,  \, \mathrm{otherwise} \\
\end{cases}
\end{equation*}
\end{prop}

\begin{proof}
$\Hom_{\overline{\OC}_{S_\al}}(\rp_{S_\al}(\mu),\rp_{S_\al}(\lam))$ is a finitely generated $S_{\al}$-module. Let $\K$ be the residue field and $Q$ the quotient field of $S_\al$. By Lemma \ref{lem3} we get
\begin{equation*}
\Hom_{\overline{\OC}_{S_\al}}(\rp_{S_\al}(\mu),\rp_{S_\al}(\lam))\otimes_{S_\al} \K \cong  \Hom_{\overline{\OC}_{\K}}(\rp_{\K}(\mu),\rp_{\K}(\lam))
\end{equation*}
and
\begin{equation*}
\Hom_{\overline{\OC}_{S_\al}}(\rp_{S_\al}(\mu),\rp_{S_\al}(\lam))\otimes_{S_\al} Q \cong  \Hom_{\overline{\OC}_{Q}}(\rp_{Q}(\mu),\rp_{Q}(\lam))
\end{equation*}

But by \cite{6}, chapter 6.2, the dimension of the right hand side of the first equality is given by
\begin{equation*}
\mathrm{dim}_{\K} \Hom_{\overline{\OC}_{\K}}(\rp_{\K}(\mu),\rp_{\K}(\lam))=
\begin{cases}
1, \mathrm{if} \, \mu = \al \downarrow \lam \, \mathrm{or} \, \mu =\al \uparrow \lam \\
2, \mathrm{if} \, \mu = \lam \\
0, \, \mathrm{otherwise} \\
\end{cases}
\end{equation*}
Since we have a decomposition $\rp_{S_\al}(\mu)\otimes_{S_\al} Q \cong \rdel_Q(\al \uparrow \mu) \oplus \rdel_Q(\mu)$ and $\rp_{S_\al}(\lam)\otimes_{S_\al} Q \cong \rdel_Q(\al \uparrow \lam) \oplus \rdel_Q(\lam)$ into simple Verma modules we also get
\begin{equation*}
\mathrm{dim}_{Q} (\Hom_{\overline{\OC}_{S_\al}}(\rp_{S_\al}(\mu),\rp_{S_\al}(\lam))\otimes_{S_\al} Q) =
\begin{cases}
1, \mathrm{if} \, \mu = \al \downarrow \lam \, \mathrm{or} \, \mu =\al \uparrow \lam \\
2, \mathrm{if} \, \mu = \lam \\
0, \, \mathrm{otherwise} \\
\end{cases}
\end{equation*}
But then by \cite{3}, Lemma 3.7 (involving Nakayama's Lemma), we get that $\Hom_{\overline{\OC}_{S_\al}}(\rp_{S_\al}(\mu),\rp_{S_\al}(\lam))$ is a free $S_\al$-module with
\begin{equation*}
\mathrm{rk}_{S_\al} \Hom_{\overline{\OC}_{S_\al}}(\rp_{S_\al}(\mu),\rp_{S_\al}(\lam))=
\begin{cases}
1, \mathrm{if} \, \mu = \al \downarrow \lam \, \mathrm{or} \, \mu =\al \uparrow \lam \\
2, \mathrm{if} \, \mu = \lam \\
0, \, \mathrm{otherwise} \\
\end{cases}
\end{equation*}
\end{proof}

For fixed $\lam \in \Gamma$ we want to give bases for the following $S_\al$-modules:
$$\End_{\overline{\OC}_{S_\al}}(\rp_{S_\al}(\lam))$$
$$\Hom_{\overline{\OC}_{S_\al}}(\rp_{S_\al}(\lam),\rp_{S_\al}(\al \uparrow \lam))$$
$$\Hom_{\overline{\OC}_{S_\al}}(\rp_{S_\al}(\lam),\rp_{S_\al}(\al \downarrow \lam))$$

We follow the notation of \cite{6}. Clearly, we can take the identity in $\End_{\overline{\OC}_{S_\al}}(\rp_{S_\al}(\lam))$ as the first basis element. Over the residue field $\K$ we have a composition
$$\rp_\K(\lam) \twoheadrightarrow \rdel_\K(\lam) \hookrightarrow \rdel_\K(\al \uparrow \lam)$$
We can lift this composition via base change to a map $\rp_{S_\al} (\lam) \rightarrow \rdel_{S_\al}(\al \uparrow \lam)$ which is unequal to $0$. Composing this map with the inclusion $\rdel_{S_\al}(\al \uparrow \lam) \hookrightarrow \rp_{S_\al}(\lam)$ yields an endomorphism $n_\lam : \rp_{S_\al}(\lam) \rightarrow \rp_{S_\al}(\lam)$ which is unequal to $0$ and after applying $\cdot \otimes_{S_\al} \K$ corresponds to the composition
$$n_\lam^\K:\rp_\K(\lam) \twoheadrightarrow \rdel_\K(\lam) \hookrightarrow \rdel_\K(\al \uparrow \lam) \hookrightarrow \rp_\K(\lam)$$
Since the identity and $n_\lam$ are linearly independent, $\{\mathrm{id}, n_\lam\}\subset \End_{\overline{\OC}_{S_\al}}(\rp_{S_\al}(\lam))$ is a basis.\\
Taking the map $\rp_{S_\al}(\lam) \rightarrow \rdel_{S_\al}(\al \uparrow \lam)$ from above and the projectivity of $\rp_{S_\al}(\lam)$ we get a morphism $b_\lam: \rp_{S_\al}(\lam) \rightarrow \rp_{S_\al}(\al \uparrow \lam)$ from the following diagram

\begin{displaymath}
    \xymatrix{
        \rp_{S_\al}(\lam) \ar[dr]\ar@{.>}[r] & \rp_{S_\al}(\al \uparrow \lam) \ar@{->>} [d]\\
              & \rdel_{S_\al}(\al \uparrow \lam) }
\end{displaymath}

This morphism is unequal to $0$ after applying $\cdot\otimes_{S_\al} \K$. We conclude that $b_\lam$ is a basis of $\Hom_{\overline{\OC}_{S_\al}}(\rp_{S_\al}(\lam),\rp_{S_\al}(\al \uparrow \lam))$. \\
Finally, we have the composition 
$$a_\lam : \rp_{S_\al}(\lam) \twoheadrightarrow \rdel_{S_\al}(\lam) \hookrightarrow \rp_{S_\al}(\al \downarrow \lam)$$
which is unequal to $0$ after base change $\cdot\otimes_{S_\al} \K$ with the residue field. Thus it is a basis of $\Hom_{\overline{\OC}_{S_\al}}(\rp_{S_\al}(\lam),\rp_{S_\al}(\al \downarrow \lam))$

\begin{remark}
Since $\al \uparrow \lam \geq \lam$, the endomorphism $n_\lam : \rp_{S_\al}(\lam) \rightarrow \rp_{S_\al}(\lam)$ restricts to an endomorphism $\rdel_{S_\al}(\al\uparrow \lam) \rightarrow \rdel_{S_\al}(\al\uparrow \lam)$ and thus induces an endomorphism on the cokernel $\rdel_{S_\al}(\lam)$ of the inclusion $\rdel_{S_\al}(\al\uparrow \lam) \hookrightarrow \rp_{S_\al}(\lam)$.
\end{remark}

\begin{prop}\label{prop1}
Up to an invertible element of $S_\al$, $n_\lam$ induces the map $\al^\vee \cdot\mathrm{id}$ on $\rdel_{S_\al}(\al \uparrow \lam)$ and the zero map on $\rdel_{S_\al}(\lam)$.
\end{prop}

Before we prove this proposition we repeat a result of \cite{10} about the Jantzen filtration on $\rdel_\K(\lam)$. \\
The Shapovalov form for Kac-Moody algebras introduced in \cite{8} induces a contravariant form $(\cdot, \cdot)_{S_\al}$ on $\rdel_{S_\al}(\mu)$ for any $\mu \in \Gamma$. We define a filtration on $\rdel_{S_\al}(\mu)$ by setting 
$$\rdel_{S_\al}(\mu)^i :=\{m \in \rdel_{S_\al}(\mu)\,|\,(m, \rdel_{S_\al}(\mu))_{S_\al} \in (\al^\vee)^i\cdot S_\al\}$$

We then get the Jantzen filtration on $\rdel_\K(\mu)$ by 
$$\rdel_\K(\mu)^i:=\mathrm{im}(\rdel_{S_\al}(\mu)^i \hookrightarrow \rdel_{S_\al}(\mu) \twoheadrightarrow \rdel_\K(\mu))$$
where $\rdel_{S_\al}(\mu) \twoheadrightarrow \rdel_\K(\mu)$ is the map induced by $\cdot \otimes_{S_\al} \K$.

\begin{lemma}[\cite{10}, Proposition 5.6]\label{lemjf}
The Jantzen filtration on $\rdel_\K(\mu)$ is given by 
$$\rdel_\K(\mu) \supset L_\K(\al \downarrow \mu)  \supset 0$$
\end{lemma}

\begin{proof}
We give a proof which is adapted from \cite{9}, chapter 5.14, and slightly different to the one given in \cite{10}.\\
Since, by construction, $\rdel_\K(\mu)^1$ coincides with the maximal submodule of $\rdel_\K(\mu)$, we get, by Theorem \ref{thm3}, that $\rdel_\K(\mu)^1\cong L_\K(\al \downarrow\mu)$. We have to prove $\rdel_\K(\mu)^2=0$. Let $m \in \rdel_\K(\mu)^1_\mu$ be a generator of highest weight $\mu$. We just have to prove $m \notin \rdel_\K(\mu)^2$. Let us assume $m \in \rdel_\K(\mu)^2$. Thus, there is an element $m' \in \rdel_{S_\al}(\mu)^2$ such that $m' \mapsto m$ under specialization $\rdel_{S_\al}(\mu) \twoheadrightarrow \rdel_\K(\mu)$. Since for $\nu > \al \downarrow \mu$ we have $\rdel_\K(\mu)_\nu^1=0$, we conclude
$$\rdel_{S_\al}(\mu)_\nu^1 \subset (\al^\vee)\cdot \rdel_{S_\al}(\mu)_\nu$$
and
$$\rdel_{S_\al}(\mu)_\nu^2 \subset (\al^\vee)^2\cdot \rdel_{S_\al}(\mu)_\nu$$

The generalized Casimir operator $C$ can be split into a sum $C=C_1 + C_2$, where $C_1(\rdel_{S_\al}(\mu)_\nu) \subset \bigoplus_{\eta>\nu} \rdel_{S_\al}(\mu)_\eta$, $C_2(\rdel_{S_\al}(\mu)_\nu)\subset \rdel_{S_\al}(\mu)_\nu$ and $C_2$ acts on $\rdel_{S_\al}(\mu)_\nu$ by multiplication with $(\nu+\tau+\rho|\nu+\tau+\rho)_{S_\al} - (\rho|\rho)_{S_\al} \in S_\al$. Since $C$ commutes with the $\wg$-action, we get that $C$ acts on $\rdel_{S_\al}(\mu)$ by multiplication with $(\mu+\tau+\rho|\mu+\tau+\rho)_{S_\al} - (\rho|\rho)_{S_\al}$. Thus, applying $C$ to $m'$ yields
$$Cm'= ((\mu+\tau+\rho|\mu+\tau+\rho)_{S_\al} - (\rho|\rho)_{S_\al})m'$$
on the one hand, but also
$$Cm'\in ((\al \downarrow \mu+\tau+\rho|\al \downarrow \mu+\tau+\rho)_{S_\al} - (\rho|\rho)_{S_\al})m' + (\al^\vee)^2\cdot\rdel_{S_\al}(\mu).$$
Therefore, by $\wW$-invariance of $(\cdot|\cdot)_{S_\al}$ and a little calculation we get
$$(\mu -\al \downarrow \mu|\tau)m'\in (\al^\vee)^2\cdot\rdel_{S_\al}(\mu).$$
Since $\mu -\al \downarrow \mu$ is either equal to $n \al$ or $n(-\al+\delta)$ for $n>0$ an integer and since $(\delta|\tau)_{S_\al}=0$, we get $(\mu-\al \downarrow \mu|\tau)_{S_\al}=k \al^\vee$ with $k \in \C\backslash \{0\}$. But then specializing $\al^\vee \mapsto 0$ yields $km=0$ which is a contradiction. Thus $m \notin \rdel_\K(\mu)^2$.
\end{proof}

With this lemma, we can now prove Proposition \ref{prop1}

\begin{proof}[Proof of Proposition \ref{prop1}]
Over the residue field $\K$, we get the following diagram of short exact sequences in the horizontals

\begin{displaymath}
    \xymatrix{
        \rdel_{\K}(\al \uparrow \lam) \ar@{^{(}->}[r] \ar[ddd]_{y \cdot \mathrm{id}} & \rp_\K(\lam) \ar@{->>}[r] \ar@{->>}[d]& \rdel_{\K}(\lam) \ar[ddd]^{{x \cdot \mathrm{id}}} \\
                                     & \rdel_{\K}(\lam) \ar@{^{(}->}[d]          & \\
                                     & \rdel_{\K}(\al \uparrow \lam) \ar@{^{(}->}[d]          & \\
       \rdel_{\K}(\al \uparrow \lam)     \ar@{^{(}->}[r]       & \rp_\K(\lam)    \ar@{->>}[r]          & \rdel_{\K}(\lam)       }
\end{displaymath}
where the composition $n_\lam^\K$ in the middle is induced by $n_\lam$ and $x,y\in \K$.
Since the composition $\rdel_{\K}(\al \uparrow \lam) \hookrightarrow \rp_\K(\lam) \twoheadrightarrow \rdel_{\K}(\lam)$ is zero, we get that both scalars, $x$ and $y$, are zero.\\
Over $S_\al$ the composition 
$$\rp_{S_\al} (\lam) \rightarrow \rdel_{S_\al}(\al \uparrow \lam) \hookrightarrow \rp_{S_\al} (\lam) \twoheadrightarrow \rdel_{S_\al}(\lam)$$
is zero, so $n_\lam$ induces the zero map on $\rdel_{S_\al}(\lam)$. But the composition $\rdel_{S_\al}(\al \uparrow \lam) \hookrightarrow \rp_{S_\al} (\lam) \rightarrow \rdel_{S_\al}(\al \uparrow \lam)$ is unequal to zero since otherwise we would have a factorization over the cokernel, in formulas
\begin{displaymath}
    \xymatrix{
        \rdel_{S_\al}(\al \uparrow \lam) \ar@{^{(}->}[r]  & \rp_{S_\al} (\lam) \ar@{->>}[dr] \ar[r]  & \rdel_{S_\al}(\al \uparrow \lam)\\
      && \rdel_{S_\al}(\lam) \ar@{.>}[u]_{\exists !}       }
\end{displaymath}
But $\rdel_{S_\al}(\lam) \rightarrow \rdel_{S_\al}(\al \uparrow \lam)$ is the zero map, because it is zero after applying $\cdot \otimes_{S_\al} Q$, while $\rp_{S_\al} (\lam) \rightarrow \rdel_{S_\al}(\al \uparrow \lam)$ is unequal to zero.\\
Now, after multiplying $n_\lam$ with an appropriate invertible element of $S_\al$, $n_\lam$ induces $(\al^\vee)^n \cdot \mathrm{id}$ on $\rdel_{S_\al}(\al \uparrow \lam)$. We want to show $n=1$. \\
Assume $n\geq 2$. Then $\mathrm{im}(n_\lam) \subset \rdel_{S_\al}(\al \uparrow \lam)^2\subset \rp_{S_\al}(\lam)$ by definition of the Jantzen filtration. But after base change $\cdot \otimes_{S_\al} \K$, we have $\mathrm{im}(n_\lam^\K) \subset \rdel_{S_\al}(\al \uparrow \lam)^2 \otimes_{S_\al} \K$. But $\rdel_{S_\al}(\al \uparrow \lam)^2 \otimes_{S_\al} \K$ is the second step of the Jantzen filtration on $\rdel_{\K}(\al \uparrow \lam)$ which is zero by Lemma \ref{lemjf}. But this contradicts $n_\lam^\K \neq 0$ and we conclude $n=1$.
\end{proof}

Let us fix a short exact sequence $\rdel_{S_\al}(\al \uparrow \lam) \hookrightarrow \rp_{S_\al}(\lam) \twoheadrightarrow \rdel_{S_\al}(\lam)$ and let $n_\lam : \rp_{S_\al}(\lam) \rightarrow \rp_{S_\al}(\lam)$ be the map from above normalized, s.t. it induces $\al^\vee \cdot \mathrm{id}$ on $\rdel_{S_\al}(\al \uparrow \lam)$ and $0$ on $\rdel_{S_\al}(\lam)$.

\begin{lemma}
\begin{enumerate}
\item We have $n_\lam \circ a_{\al \uparrow \lam} = a_{\al \uparrow \lam} \circ (\al^\vee \cdot \mathrm{id}-n_{\al \uparrow \lam})$
\item We have $(\al^\vee \cdot \mathrm{id}-n_{\al \uparrow \lam})\circ b_\lam = b_\lam \circ n_\lam$
\end{enumerate}
\end{lemma} 

\begin{proof}
\begin{enumerate}
\item This part of the lemma is clear by definition of the map $n_\lam$ and the effect it has on restricted Verma modules.
\item Applying $\cdot \otimes_{S_\al} Q$ to the diagram
\begin{displaymath}
    \xymatrix{
        \rp_{S_\al}(\lam) \ar[r]^{b_\lam} \ar[d]_{n_\lam} & \rp_{S_\al}(\al \uparrow \lam) \ar[d]^{\al^\vee \cdot \mathrm{id} -n_{\al \uparrow \lam}} \\
       \rp_{S_\al}(\lam) \ar[r]_{b_\lam}       & \rp_{S_\al}(\al \uparrow \lam) }
\end{displaymath}
identifies with the diagram
\begin{displaymath}
    \xymatrix{
        \rdel_{Q}(\lam) \oplus \rdel_{Q}(\al \uparrow \lam)\ar[r] \ar[d]_{g_1} & \rdel_Q(\al \uparrow^2 \lam) \oplus \rdel_Q(\al \uparrow \lam) \ar[d]^{g_2} \\
       \rdel_Q(\lam) \oplus \rdel_Q(\al \uparrow \lam) \ar[r]      & \rdel_Q(\al \uparrow^2 \lam) \oplus \rdel_Q(\al \uparrow \lam) }
\end{displaymath}
where $g_1, g_2$ are both given by the matrix
$$ \begin{pmatrix} 0 & 0\\ 0 & (\al^\vee \cdot) \end{pmatrix}  $$
But this diagram certainly commutes. Then the diagram above over $S_\al$ commutes as well.
\end{enumerate}
\end{proof}

By the description of the center in Remark \ref{rmk2}, and by what we have discovered above, the image of the inclusion
$$\mathcal{Z}_{S_\al}(\Gamma) \hookrightarrow \prod_{\substack{\lam \in \Gamma}} \End_{\overline{\OC}_{S_\al}}(\rp_{S_\al}(\lam))$$
is generated by the tuples $(\mathrm{id}_{\rp_{S_\al}(\mu)})_{\mu \in \Gamma}$ and $\{(\delta_\lam^\mu)_{\mu \in \Gamma}\}_{\lam \in \Gamma}$ where $$\delta_\lam^\mu= \begin{cases} n_\lam, \,\mathrm{if} \, \mu=\lam \\ \al^\vee \cdot \mathrm{id} - n_{\al \uparrow \lam}, \,\mathrm{if} \, \mu=\al \uparrow \lam \\ 0, \,\mathrm{otherwise} \end{cases}.$$ 
But the images of these generators under the map
$$\phi: \prod_{\substack{\mu \in \Gamma}} \End_{\overline{\OC}_{S_\al}}(\rp_{S_\al}(\mu)) \rightarrow \prod_{\substack{\mu \in \Gamma}} \End_{\overline{\OC}_{S_\al}}(\rdel_{S_\al}(\mu)) \cong \prod_{\substack{\mu \in \Gamma}} S_\al$$
are $\phi((\mathrm{id}_{\rp_{S_\al}(\mu)})_{\mu \in \Gamma}) = (1)_{\mu \in \Gamma}$ and $\phi((\delta_\lam^\mu)_{\mu \in \Gamma})= (\kappa_\lam^\mu)_{\mu \in \Gamma}$ with\\
$\kappa_\lam^\mu= \begin{cases} \al^\vee, \mathrm{if}\, \mu = \al \uparrow \lam \\ 0, \mathrm{else} \end{cases}$.
As a conclusion we have 
 
\begin{prop}\label{propcent}
$$\mathcal{Z}_{S_\al}(\Gamma) \cong \left \{(z_\mu)_{\mu \in \Gamma} \in \prod_{\mu \in \Gamma} S_\al \middle | z_\mu \equiv z_{\al \uparrow \mu} \,(\mathrm{mod} \, \al^\vee) \right \}$$
\end{prop}

\subsection{The general case}

In this chapter we want to carry together our results to prove the main theorem. Let $\Lam$ be an equivalence class under $\sim_{\st}^\mathrm{res}$.

\begin{thm}
$$\mathcal{Z}_{\st}(\Lam) \cong \left \{(z_\mu)_{\mu \in \Lam} \in \prod_{\substack{\mu \in \Lam}} \st \middle | z_\mu \equiv z_{\al \downarrow \mu} \, (\mathrm{mod} \, \al^\vee) \, \forall \al \in R_{\st}(\Lam) \right \}$$
\end{thm}

\begin{proof}
For $\p \subset \st$ a prime ideal of height one and for $Q=\mathrm{Quot}(\st)$, we have a base change map
$$\mathcal{Z}_{\st}(\Lam) \hookrightarrow  \mathcal{Z}_{\st}(\Lam) \otimes_{\st} S_{\p} \cong \mathcal{Z}_{S_\p}(\Lam) \subset \mathcal{Z}_Q(\Lam)\cong \prod_{\mu \in \Lam}Q.$$
We also have 
$$\mathcal{Z}_{\st}(\Lam) = \bigcap_{\substack{\p \in \mathfrak{P}}} \mathcal{Z}_{S_\p}(\Lam)$$
in the $Q$-vector space $\prod_{\substack{\mu \in \Lam}} Q$. Here, $\mathfrak{P}$ denotes the set of all prime ideals of $\st$ of height one. If $\al^\vee \notin \p$ for all $\al \in R_{\st}(\Lam)$, then all Verma modules are projective and we have $\mathcal{Z}_{S_\p}(\Lam)\cong \prod_{\mu \in \Lam}S_\p$.\\
If $\p$ is generated by $\al^\vee$ for $\al \in R_{\st}(\Lam)$, $\mathcal{Z}_{S_\p}(\Lam)$ decomposes into the direct sum of modules of the form described in Proposition \ref{propcent} and in modules of the form $\prod_{\substack{\mu \in \Gamma}} S_\p$, where $\Gamma \subset \Lam$ is a generic equivalence class under $\sim_{S_\p}^\mathrm{res}$, i.e., $\Gamma$ only contains one element. This proves the claim.
\end{proof}

\section{Acknowledgements}
I would like to thank my supervisor Peter Fiebig for many inspiring discussions.

\bibliographystyle{amsplain}

\end{document}